\documentclass[11pt]{amsart}

\usepackage[a4paper,hmargin=3.5cm,vmargin=4cm]{geometry}
\usepackage{amsfonts,amssymb,amscd,amstext}
\usepackage{graphicx}
\usepackage[dvips]{epsfig}
\usepackage{quoting}
\usepackage{hyperref}

\usepackage{fancyhdr}
\usepackage{enumitem}
\usepackage{titlesec}
\usepackage{mathrsfs}

\titleformat{\section}
{\filcenter\bfseries} {\thesection{.}}{0.2cm}{}
\titleformat{\subsection}[runin]
{\bfseries} {\thesubsection{.}}{0.15cm}{}[.]
\titleformat{\subsubsection}[runin]
{\em}{\thesubsubsection{.}}{0.15cm}{}[.]

\usepackage[up,bf]{caption}

\newtheorem{theorem}{Theorem}[section]

\newtheorem{lemma}[theorem]{Lemma}

\newtheorem{corollary}[theorem]{Corollary}

\theoremstyle{definition}

\newtheorem{remark}[theorem]{Remark}

\numberwithin{equation}{section}
\numberwithin{figure}{section}

\newcommand\Dcal{\mathcal{D}}

\newcommand\Ascr{\mathscr{A}}

\newcommand\Cscr{\mathscr{C}}

\newcommand\Oscr{\mathscr{O}}

\newcommand\C{\mathbb{C}}

\newcommand\N{\mathbb{N}}

\newcommand\R{\mathbb{R}}

\renewcommand\S{\mathbb{S}}

\newcommand\Z{\mathbb{Z}}

\newcommand\h{\mathbb{H}}
\newcommand\n{\mathbb{N}}

\newcommand\igot{\mathfrak{i}}

\renewcommand\igot{\mathfrak{i}}

\newcommand\wt{\widetilde}

\newcommand{\SL}{{\rm{SL}}_2(\C)}

\newcommand{\UH}{T_1 \mathbb{H}^3}
\newcommand{\UR}{T_1 R}
\newcommand{\PSL}{{\rm{PSL}}_2(\C)}
\newcommand{\id}{{\rm Id}}
\newcommand{\U}{{\rm{U}_1}}

\newcommand{\om}{\overline{M}}

\usepackage{color}

\usepackage{array}

\begin{document}
	
\fancyhead[LO]{ Properly embedded Legendrian surfaces in $\UH$}
\fancyhead[RE]{J. Hidalgo}
\fancyhead[RO,LE]{\thepage}
	
\thispagestyle{empty}
	

\begin{center}
	{\bf\Large Properly embedded Legendrian surfaces in the tangent sphere bundle of hyperbolic space
			} 
		
	\medskip
%
%
{\bf Jorge Hidalgo}
\end{center}
	
\medskip
	
\begin{quoting}[leftmargin={5mm}]
{\small
	\noindent {\bf Abstract}\hspace*{0.1cm}
	We show that every open Riemann surface may be conformally properly embedded in the tangent sphere bundle of hyperbolic $3$-space as a Legendrian surface. We achieve this by constructing proper holomorphic Legendrian embeddings in $\SL$ whose projections to the quotient $\SL / \U$ remain embeddings. As another consequence, every open Riemann surface admits a proper holomorphic Legendrian embedding in the projectivised special linear group $\PSL$.

	\noindent{\bf Keywords}\hspace*{0.1cm} 
	 Riemann surface, Legendrian surface, holomorphic Legendrian curve, Tangent sphere bundle, Special linear group, Unitary subgroup

	\noindent{\bf Mathematics Subject Classification (2020)}\hspace*{0.1cm} 
	53D10, 
	32H02, 
	32E30, 
	53C42.
}
\end{quoting}
\section{Introduction}\label{sec:intro}
\noindent

Let $(R, \langle \cdot, \cdot \rangle)$ be a $3$-dimensional Riemannian manifold. We consider the {\em tangent sphere bundle} $\UR$ as the subbundle  of the tangent bundle $TR$ consisting of pairs $(x, v)$ such that $x \in R$, $v \in T_x R$, and $\langle v, v \rangle = 1$. This 5-dimensional manifold is canonically endowed with the {\em Sasaki metric} and with a {\em contact structure} induced by $\langle \cdot, \cdot \rangle$, rendering $\UR$ a contact Riemannian manifold.
Tangent sphere bundles are classical objects in differential geometry, both for their relationship with the geodesic flow \cite{Besse1978} and for the interplay between its Riemannian and contact structures. We refer to the pioneering works of Sasaki \cite{Sasaki1958, Sasaki1962} and the surveys \cite{BoeckxVanhecke2000, Calvaruso2005} for background and historical developments; see also \cite{Wolf1974, MussoTricerri1988, Blair2010}.
A {\em Legendrian surface} in $\UR$ is an immersion $L : M \to \UR$ from a smooth surface $M$ that is tangent to the contact structure of $\UR$. 
These surfaces arise naturally when we consider a smooth map $f : M \to R$ that admits a unit normal vector field $\nu$ and form the Legendrian immersion $L = (f, \nu) : M \to \UR$. Conversely, every Legendrian immersion $L : M \to \UR$ projects to a smooth map $f: M \to R$ after composing with the natural projection $TR \to R$, which may have singular points where the differential vanishes. Thus Legendrian surfaces in $\UR$ may be thought of as a generalization of regular surfaces in $R$.

In this paper we focus on constructing properly embedded Legendrian surfaces in the tangent sphere bundle of hyperbolic $3$-space $\h^3$. 
%
%
%
\begin{theorem}\label{th:mt}
Every open Riemann surface admits a proper conformal Legendrian embedding in $\UH$.
\end{theorem}
Legendrian surfaces in $\UH$ are particularly interesting among Legendrian surfaces in tangent sphere bundles as they connect the holomorphic representation of flat surfaces by Gálvez, Martínez, and Milán \cite{GalvezMartinezMilan2000MA} with its extension to flat fronts (i.e., flat surfaces with admissible singularities) by Kokubu, Umehara, and Yamada \cite{KokubuUmeharaYamada2004}; see also \cite{KRSUY2005, KRUY2007,  GalvezMira2005CVPDE}. 
Regarding other results on Legendrian surfaces in tangent sphere bundles of $3$-dimensional space forms, Li, Wang, and  Wang  \cite{LiWang2024, LiWang2024PAMS} established some classification theorems in 2024 for minimal Legendrian surfaces in the tangent sphere bundles of Euclidean space $\R^3$ and the sphere $\S^3$ under strong restrictions on the curvature or the topology. Even more recently, Li proved a classification result for complete Legendrian surfaces with constant mean curvature in $T_1 \R^3$ in \cite{Li2026}. No analogous results seem to be available yet for the hyperbolic case.

Note that an equivalent way of rephrasing Theorem \ref{th:mt} is to say that every open Riemann surface is the underlying complex structure of a properly embedded Legendrian surface in $\UH$. 
To the best of the author's knowledge, this theorem provides the first construction of properly embedded Legendrian surfaces in a tangent sphere bundle with prescribed complex structure.
This result is similar to the fact that every open Riemann surface admits a proper conformal constant mean curvature one immersion in $\h^3$,
see \cite{AlarconHidalgo2026}.  A major difference is that now we work with surfaces immersed in the higher dimensional manifold $\UH$, so we have larger codimension. This is what allows us to apply a general position argument and obtain embedded surfaces, see below for more.

To prove Theorem \ref{th:mt} we consider the hermitian model of $\mathbb{L}^4$ and identify 
\[
\UH =   \{ (A \overline{A}^T, A e_3 \overline{A}^T) \colon A \in \SL \} \equiv \SL / \U, \quad e_3 =  \left(\begin{array}{cc}
	1 & 0\\
	0 &  -1
\end{array}\right),
\]
where $\SL$ denotes the {\em special linear group}
\[
\SL = \left\{  z=   \left( \begin{array}{cc}
	z_{11} & z_{12}\\
	z_{21}& z_{22} 
\end{array} \right) \in \C^4 \colon  \det z =   z_{11} z_{22} - z_{12} z_{21}  = 1 \right\},
\] 
\begin{equation}\label{eq:u1}
	\rm{U}_1 = \left\{ A_{\theta} = \left( \begin{array}{cc}
		e^{\igot \theta} & 0 \\
		0 & e^{-\igot \theta}
	\end{array}\right) \colon \theta \in \R
	\right\} \equiv \mathbb{S}^1 \equiv   [0, 2 \pi] / \{ 0, 2 \pi\}
\end{equation}
the {\em unitary subgroup} of $\SL$, and $\igot = \sqrt{-1}$.
The canonical projection 
\[
\begin{split}
\pi : 	\SL & \to  \UH  \equiv \SL / \U\\
	A & \mapsto (A\overline{A}^T , A e_3  \overline{A}^T)
\end{split}
\]
is a proper map which takes {\em holomorphic Legendrian curves} in $\SL$ into conformal Legendrian immersions in $\UH$; note that $\U$ is compact and see Subsection \ref{subsec:leguh3} for details. Recall that a holomorphic Legendrian curve $F : M \to \SL$ from an open Riemann surface $M$ is a holomorphic immersion which is tangent to the standard holomorphic contact form of $\SL$ given by
 \[
\eta_{\rm{SL}} = z_{11} d z_{22} - z_{21} d z_{12}.
\]
To this respect, Alarcón showed in 2017 that every open Riemann surface $M$ admits a proper holomorphic Legendrian embedding $M \to \SL$, see \cite[Theorem 1.1]{Alarcon2017JGA}. 
This already yields a proper conformal Legendrian immersion $M \to \UH$ for every open Riemann surface $M$. 
However, we emphasize that being embedded in $\SL$ is clearly not enough to obtain embeddings when passing to the quotient $\UH \equiv \SL / \U$.

To deal with this we define the map
\[
\begin{split}
	\Delta F : \R \times M  \times M & \to \C^4\\
	(\theta, p, q) &  \mapsto  \Delta F (\theta, p, q) = F(q) -F(p)  A_{\theta}.
\end{split}
\]
We prove the following. 
%
%
\begin{theorem}\label{th:mtleg}
Every open Riemann surface $M$ admits a proper holomorphic Legendrian curve $F : M \to \SL$ such that 
\[\Delta F (\theta, p, q) = 0
\]
 if and only if $\theta \in 2 \pi \Z$ and $p = q$.
\end{theorem}
Note that a curve $F$ satisfying the conclusion of Theorem \ref{th:mtleg} is in particular injective. Since a proper injective immersion in $\SL$ is an embedding, this result generalizes  \cite[Theorem 1.1]{Alarcon2017JGA}. Theorem \ref{th:mt} follows as an immediate corollary of Theorem \ref{th:mtleg} after projecting the curves to the quotient $\UH \equiv \SL / \U$. 
We remark that since the Legendrian embeddings $M \to \UH$ obtained in Theorem \ref{th:mt} are projections of proper holomorphic Legendrian embeddings $M \to \SL$, they project to proper, weakly complete, flat fronts $M \to \h^3$ whose canonical complex structure agrees with that of $M$; see \cite{KokubuUmeharaYamada2004}.

Let us mention another consequence of Theorem \ref{th:mtleg}. 
Consider the quotients $\SL / C_n$, where 
\[
C_n = \{ A_{2 \pi k/n} \colon 1 \le k \le n, \, k \in \N\}, \quad n \in \N  =  \{ 1, 2, 3, \hdots\},
\]
denote the cyclic subgroups of $\SL$, see \eqref{eq:u1}.  These are complex contact manifolds with the contact structure induced from $\SL$. 
Note that $\SL /C_2$ corresponds to the projectivisation $\PSL = \SL / \{\pm \id\}$. Emerging also in the analytic representation of flat surfaces in $\h^3$, holomorphic Legendrian curves in $\PSL$ are of great interest as they provide an integral-free formula to describe these surfaces, see \cite{GalvezMartinezMilan2000MA, KokubuUmeharaYamada2003}. 
We obtain the following immediate consequence of Theorem \ref{th:mtleg} after projecting the holomorphic Legendrian curves in $\SL$ to the different quotients $\SL/C_n$. 
%
%
\begin{corollary}
Every open Riemann surface admits a proper holomorphic Legendrian embedding in $\SL /C_n$ for every $n \in \N$. In particular, every open Riemann surface admits a proper holomorphic Legendrian embedding in $\PSL$.
\end{corollary}

\textit{Strategy to prove Theorem \ref{th:mtleg}.} As in \cite{Alarcon2017JGA}, our approach relies on the approximation theory for holomorphic Legendrian curves in $\C^3$ developed by Alarcón, Forstneri\v c, and López in \cite{AlarconForstnericLopez2017CM}, where $\C^3$ is endowed with its standard holomorphic contact form
\[
\eta_0 = d z_3 + z_1 d z_2,
\]
and $(z_1, z_2, z_3)$ denote the complex coordinates on $\C^3$.
We shall use the holomorphic map $\Dcal : \C^3 \to \SL$ introduced by Martín, Umehara, and Yamada \cite{MartinUmeharaYamada2014} and given by
\begin{equation}\label{eq:d}
	\Dcal(z_1, z_2, z_3)=\left(\begin{matrix}
		e^{-z_3} & z_1 e^{z_3}\\
		z_2 e^{-z_3} &  (1 + z_1 z_2) e^{z_3}
	\end{matrix}\right),\quad (z_1,z_2,z_3)\in \C^3,
\end{equation}
which carries holomorphic Legendrian curves $M \to \C^3$ into holomorphic Legendrian curves $M \to \SL$.
We remark that $\Dcal$ is neither proper nor injective, see \cite[Remark 4.1]{Alarcon2017JGA}. 

Consider an open Riemann surface $M$. We shall use holomorphic sprays and a transversality argument as in \cite[Proof of Lemma 4.3]{AlarconForstnericLopez2017CM}, but adding a new real parameter $\theta \in \R$ and exploiting the fact that $\dim (\R\times  M \times M ) < \dim (\C^3)$, to construct a holomorphic Legendrian curve $X= (X_1, X_2, X_3) : M \to \C^3$ such that the map
\begin{equation}\label{eq:thetadifleg}
	\begin{split}
		\delta^{\Dcal} X : 	\R \times   M \times  M  & \to \C^3\\
		(\theta, p,q) & \mapsto  (X_1 (q) -  X_1 (p), X_2 (q) - X_2 (p), e^{- X_3 (q)}    - e^{\igot \theta} e^{- X_3 (p)}   )  \\
	\end{split}
\end{equation}
vanishes only when $\theta \in 2 \pi \Z$ and $p=q$. Having established this property, it is straightforward to check that $F = \Dcal \circ X : M \to \SL$ satisfies that $\Delta F (\theta, p ,q)= 0$ if and only if  $\theta \in 2 \pi \Z$ and $p=q$. To ensure that $F$ can be chosen proper as well, we argue as in \cite{Alarcon2017JGA}.

\textit{Organization of the paper.} In Section \ref{sec:prel} we provide some background and introduce notation and a preliminar lemma. In Section \ref{sec:genpos} we prove a general position result for holomorphic Legendrian curves in $\C^3$. This provides the main technical tool to prove Theorem \ref{th:mtleg}, a task that we deal with in Section \ref{sec:proofleg}.

%
%
\section{Background and preliminaries}\label{sec:prel}
%
%
\subsection{Legendrian surfaces in $\UH$}\label{subsec:leguh3}
Here we review the contact metric structure of $\UH$ and the canonical projection $\pi : \SL \to \UH$. 
We consider the hyperboloid model for the hyperbolic space 
\[
\h^3 = \{ x= (x_0,x_1, x_2, x_3) \in \mathbb{L}^4 \colon  \langle x, x \rangle = -x_0^2 + x_1^2 + x_2^2 + x_3^2 = -1 , \, x_0 >0\}
\]
and its tangent sphere bundle
\[
\UH = \{ (x, v) \in \mathbb{L}^4 \times \mathbb{L}^4 \colon - \langle x, x \rangle = \langle v,v \rangle = 1, \, \langle x, v \rangle = 0, \,  x_0 >0 \},
\]
where $\langle \cdot, \cdot \rangle$ is the Lorentzian metric of Minkowski space $\mathbb{L}^4$ with signature $(-+++)$. We fix global coordinates $(x, v)$ on $\UH$ and express the {\em Sasaki metric} and {\em canonical contact form} as
\[
d \sigma^2 = \langle d x, d x \rangle + \langle d v , d v \rangle + \eta^2 , \quad \eta = \langle d x, v \rangle,
\]
respectively, note that $\eta^2 = \eta \otimes \eta$ and that we view the tangent bundle $T \UH$ as a subset of $\mathbb{L}^4 \times \mathbb{L}^4$; see \cite{Calvaruso2005, Blair2010}. Given an immersion $L = (f, \nu) : M \to \UH$ from a smooth surface $M$, the pullback of the Sasaki metric and contact form are given by 
\[
 L^* d \sigma^2  =  \langle d f, d f \rangle + \langle d \nu  , d \nu \rangle  +  \langle d f, \nu \rangle^2 , \quad L^* \eta = \langle d f, \nu \rangle.
 \]
Now suppose that we have a holomorphic Legendrian curve $F : M \to \SL$ from an open Riemann surface $M$. It is well known that we can write
\[
F^{-1} d F = \left(\begin{array}{cc}
		0 & \theta\\
		\omega & 0
	\end{array}\right)
\]
for a pair of holomorphic $1$-forms $\theta$ and $\omega$ on $M$, see \cite{GalvezMartinezMilan2000MA, KokubuUmeharaYamada2004}.  Identify $\mathbb{L}^4$ with the space of $2\times2$ hermitian matrices ${\rm Her}_2$ by
\[
\mathbb{L}^4 \ni (x_0, x_1, x_2, x_3) \longleftrightarrow \sum_{k=0}^3 x_k e_k =   \left(  \begin{array}{cc}
x_0 + x_3 & x_1 + \igot x_2 \\
x_1 - \igot x_2 & x_0- x_3
\end{array} \right) \in {\rm Her}_2,
\]
where
\[
e_0 = \left(\begin{array}{cc}
1 & 0\\
0 & 1
\end{array}\right), \quad 
e_1 = \left(\begin{array}{cc}
0 & 1\\
1 & 0
\end{array}\right), \quad 
e_2 = \left(\begin{array}{cc}
0 & \igot \\
-\igot  & 0
\end{array}\right), \quad 
e_3=  \left(\begin{array}{cc}
1 & 0\\
0 & -1
\end{array}\right).
\]
The Lorentzian metric becomes 
\begin{equation}\label{eq:lorentz}
\langle A , B \rangle =  (-1/2) \mathrm{Trace}(A e_2 B^T e_2)  \text{ for all } A, B \in {\rm Her}_2,
\end{equation}
in particular $\langle A , A \rangle  = - \det (A)$. The tangent sphere bundle of $\h^3$ now is
\[
\UH =   \{ (A \overline{A}^T, A e_3 \overline{A}^T) \colon A \in \SL \} \equiv \SL / \U,
\]
recall \eqref{eq:u1}, and we have a canonical projection $\pi : \SL \to \UH$.  Set 
\[
L = \pi \circ F = (F \overline F^T ,  F e_3 \overline F^T )  : M \to \UH.
\]
By a direct computation using \eqref{eq:lorentz} and taking into account that conjugation by matrices in $\SL$ are isometries of $({\rm Her}_2, \langle \cdot, \cdot \rangle)$, we obtain that
\[
L^* d \sigma^2  = |\theta|^2 + |\omega|^2, \quad L^* \eta = 0,
\]
here $|\cdot |$ is the Euclidean norm on $\C$.
Hence the projection $\pi : \SL \to \UH$ carries holomorphic Legendrian curves into conformal Legendrian surfaces. 

%
%
\subsection{Notation}
Given a set $S$ we shall denote its {\em diagonal} by $D_S = \{ (p,p) \in S \times S\}$. 
We shall denote the Euclidean norm on $\C^n$ by $|\cdot |$, for every $n \in \N$.  
Let $M$ be an open Riemann surface,  $N$ a complex manifold, and $\varnothing\neq S \subset M$ a subset. 
We denote by $\Cscr^r(S,N)$, $r\in \{0\}\cup \N$,  the space of all maps $S\to N$ of class $\Cscr^r$ on a neighborhood of $S$ in $M$ and write $\Cscr^r (S) = \Cscr^r (S, \C)$.
If $f \in \Cscr^r (S)^n$ ($n \in \N$) and $S$ is compact, we denote by $\| f \|_{r,S}$ the $\Cscr^r$-maximum norm of $f$ in $S$.
We write $\Oscr (S, N)$ for the space of all holomorphic maps from some neighborhood of $S$ in $M$ into $N$. We set $\Oscr(S) = \Oscr(S, \C)$ and
\[
\Ascr^r (S, N) = \Cscr^r (S, N) \cap \Oscr(\mathring{S}, N) ,  \quad \Ascr^r (S)= \Ascr^r(S, \C).
\]
We shall omit the target from the notation $\Oscr (S, N)$ and $\Ascr^r (S, N)$ when it is clear from the context.
A {\em bordered Riemann surface} is an open Riemann surface $M$ which is the interior of a compact one dimensional complex manifold $\om = M \cup bM$ with smooth boundary $b M \not = \varnothing$ consisting of finitely many closed Jordan curves.
Recall that any bordered Riemann surface is biholomorphic to a relatively compact smoothly bounded domain in a larger Riemann surface \cite[Theorem 8.1]{Stout1965TAMS}.
If $S$ is a subset of an open Riemann surface $R$, we say that a map $S \to \C^3$ is a holomorphic Legendrian curve if it extends to a holomorphic Legendrian curve on an open neighborhood of $S$ in $R$.
Given  a compact bordered Riemann surface $\om$ we will say that $X : \om \to \C^3$ is a Legendrian curve of class $\Ascr^r (\om)$ if it is a holomorphic Legendrian curve in $M$ and $X \in \Ascr^r (\om, \C^3)$. 
 The same terminology will apply to Legendrian curves in $\SL$. 
We recall that a compact subset $K$ of an open Riemann surface $M$ is {\em Runge} if every function in $\Ascr ( K)$ may be approximated uniformly on $K$ by holomorphic functions defined on $M$. By the Runge-Mergelyan theorem \cite[Theorem 5]{FFW18} this is equivalent to $M \setminus K$ having no relatively compact connected components.

%
%
\subsection{Holomorphic sprays of Legendrian curves}

Let us record here the following tool to construct embedded Legendrian curves in $\C^3$, introduced in \cite[Proof of Lemma 4.4]{AlarconForstnericLopez2017CM}.
\begin{lemma}\label{lemma:fixed}
Let $\om$ be a smoothly bounded compact domain contained in an open Riemann surface and $X  : \om \to \C^3$ a holomorphic Legendrian curve. For any fixed $(p_0, q_0) \in \om \times \om \setminus D_{\om}$ there exists a holomorphic map $H  : \om \times \C^3 \to \C^3$ and an open neighborhood $\Omega$ of $0$ in $\C^3$ satisfying the following properties:
\begin{enumerate}[label = \rm (\roman*)]
	\item \label{coreA} $H(\cdot, 0) = X$.
	\item \label{legendrianA} $H (\cdot, t) : \overline M \to \C^3$ is a holomorphic Legendrian curve for all $t \in \Omega$.
	\item \label{thetadomA} The map $\delta_0 H :  \om \times \om \times \C^3 \to \C^3$ defined by 
	\[
	 \delta_0 H (p,q, t)  =  H (q, t) - H (p, t)
	\] 
	is such that the partial differential
	\[
	 \partial_{t} \big|_{t = 0}  \delta_0 H (p_0 , q_0, t) : \C^3 \to \C^3
	\]
	is surjective.

\end{enumerate}
\end{lemma}

%
%
\section{A general position result for Legendrian curves}\label{sec:genpos}
\noindent

This section is devoted to establish the following lemma.
%
%
%
\begin{lemma}\label{lemma:theta}
Let $\overline{M}$ be a compact bordered Riemann surface. Every Legendrian curve $X  : \overline M \to \C^3$ of class $\Ascr^1 (\overline{M})$ may be approximated in the $\Cscr^1 (\overline{M})$-norm by Legendrian curves $\wt X = (\wt X_1, \wt X_2, \wt X_3)  : \om \to \C^3$ of class $\Ascr^1 (\overline M)$ such that the map $\delta^{\Dcal} \wt X : [0, 2 \pi]  \times \om \times \om \to \C^3$
does not vanish on $ [0, 2 \pi] \times (\om \times \om \setminus D_{\overline M})$.
\end{lemma}
This lemma is a generalization of \cite[Lemma 3.1]{Alarcon2017JGA}, which in turn is a subtle extension of \cite[Lemma 4.4]{AlarconForstnericLopez2017CM}. It allows us to construct embedded Legendrian curves in $\C^3$ which, when composing with the map $\Dcal : \C^3 \to \SL$ given in \eqref{eq:d}, will give us embedded Legendrian curves in $\SL$ that project to embedded Legendrian curves in $\UH \equiv \SL / \U$. 
\begin{remark}\label{rk:periodic}
Note that the map $\delta^{\Dcal}$ defined in \eqref{eq:thetadifleg} is periodic on the $\theta$ variable with period $2 \pi$, hence why we restrict $\delta^{\Dcal} \wt X$ in the statement of Lemma \ref{lemma:theta} to $[0, 2 \pi]  \times \om \times \om$. 
\end{remark}
%
%
\begin{proof}
Let $R$ be an open Riemann surface containing $\om$ as a smoothly bounded compact domain which is a deformation retract of $R$. By \cite[Lemma 4.3]{AlarconForstnericLopez2017CM}, we may assume that $X$ is a holomorphic Legendrian curve on $R$. 
Set $I = [0, 2 \pi]$ and consider the closed discrete subset 
\begin{equation}\label{eq:lamba}
\Lambda = \{ (0, 0, 2 k \pi \igot) \colon k \in \Z \}  \subset \C^3.
\end{equation}
Write $X = (X_1, X_2, X_3)$, let us define the map
\begin{equation}\label{eq:theta}
	\begin{split}
		\delta  X : I \times \overline M \times \overline M & \to \C^3\\
		(\theta, p,q) & \mapsto  ( X_1 (q) -  X_1 (p),  X_2 (q) -  X_2 (p),   X_3 (q)  -  X_3 (p)  + \igot \theta   ) \\
	\end{split}
\end{equation}
and note that
\begin{equation}\label{eq:dD}
	\delta X (\theta, p ,q) \in \Lambda \text{ if and only if } \delta^{\Dcal} X (\theta, p,q) = 0, \quad (\theta, p, q ) \in I \times \om \times \om,
\end{equation}
see \eqref{eq:thetadifleg}. 
Thus Lemma \ref{lemma:theta} will be proved if we are able to find a holomorphic Legendrian curve $\wt  X = (\wt X_1, \wt X_2, \wt X_3) :  \om \to \C^3$ arbitrarily close to $X$ in the $\Cscr^1 (\om)$-norm and such that $\delta \wt X (I \times (\overline M \times \overline M \setminus D_{\overline M})) \cap \Lambda = \varnothing$; here $\delta \wt X$ is defined as in \eqref{eq:theta} but using the component functions of $\wt X$. 

Let us show that such a curve exists.  Since $X$ satisfies the Legendrian condition
\[
d X_3 + X_1 d X_2 = 0 \text{ on } \om,
\] 
the map $(X_1, X_2) : \om \to \C^2$ is an immersion. 
As every immersion is locally injective, there is an open neighborhood $D_{\om} \subset U \subset \om \times \om$ such that
\begin{equation}\label{eq:u}
		\delta X ( I \times ( U \setminus D_{\overline M})) \cap \Lambda = \varnothing,
\end{equation} 
note that $\delta X (\theta, p, q) \in \Lambda$ if its first two coordinates vanish, which does not happen for any $(\theta, p , q) \in I \times (U \setminus D_{\overline M})$ by our choice of $U$; see  \eqref{eq:lamba} and \eqref{eq:theta}.

By Lemma \ref{lemma:fixed} there exists a holomorphic map $H = (H_1, H_2, H_3) :  \overline M \times \C^3 \to \C^3$ such that the following conditions hold for an open neighborhood $\Omega$ of $0$ in $\C^3$:
\begin{enumerate}[label = \rm (\arabic*)]
	\item \label{core} $H(\cdot, 0) = X$.
	\item \label{legendrian} $H (\cdot, t) : \overline M \to \C^3$ is a holomorphic Legendrian curve for all $t \in \Omega$.
	\item \label{thetadom} The map $\delta H :I \times \om \times \om \times \C^3 \to \C^3$ defined by 
	\[
	\begin{split}
			&  \delta H (\theta, p,q, t) = \\
			&( H_1 (q, \zeta) - H_1 (p, t),  \,  H_2 (q, t) -   H_2 (p, t),   \,  H_3 (q, t) -  H_3 (p, t) + \igot \theta   ) 
		\end{split}
	\]
	(cf. \eqref{eq:theta}) is such  that the partial differential
	\begin{equation}\label{eq:thetadom}
	\partial_{t} \big|_{t = 0}  \delta H (\theta, p , q, t) : \C^3 \to \C^3
	\end{equation}
	is surjective for all $(\theta, p, q) \in I \times (\om \times \om \setminus U)$.
\end{enumerate}
Indeed, fix $(p_0, q_0) \in \om \times \om \setminus U$ and apply Lemma \ref{lemma:fixed}. This furnishes a holomorphic spray $H : \om \times \C^3 \to \C^3$ such that \ref{core} and \ref{legendrian} hold, and such that 
\[
\partial_{t} \big|_{t = 0}  \delta_0 H (p_0 , q_0, t) : \C^3 \to \C^3
\]
is surjective, recall \ref{thetadomA}. By definition of $\delta_0 H$ and $\delta H$, 
\[
\partial_{t} \big|_{t= 0} \delta_ 0 H (p_0 ,q_0, t) =  \partial_{t} \big|_{t= 0}  \delta H (\theta, p_0, q_0, t) \text{ for all $\theta \in I$.}
\]
Since the term at the left is surjective, it follows that $H : \om \times \C^3 \to \C^3$ satisfies \eqref{eq:thetadom} for every triple $(\theta, p_0 ,q_0)$ with $\theta \in I$. Being a submersion is an open property, so there is an open neighborhood $(p_0,q_0) \in U_0  \subset \om \times \om$ such that $H : \om \times \C^3 \to \C^3$ satisfies \ref{core}, \ref{legendrian}, and \ref{thetadom} on $I \times U_0$. Since $\om \times \om \setminus U$ is compact, we obtain a holomorphic map that satisfies \ref{core}--\ref{thetadom} by composing finitely many such sprays as explained in \cite[Proof of Theorem 4.1]{AlarconForstnericLopez2016MZ} or \cite[Proof of Theorem 2.4]{AlarconForstneric2014IM}.

Now we can conclude the proof of Lemma \ref{lemma:theta}. By \ref{thetadom} and since  $I \times (\om \times \om \setminus U)$ is compact, there exists an open subset $0 \in \Omega' \subset \Omega \subset \C^N$ such that $	\partial_{t} \big|_{t = 0}  \delta H$ is surjective on  $I \times  (\om \times \om \setminus U) \times \Omega'$. By Abraham's transversality argument using Sard's lemma (see \cite{Abraham1963} or \cite[Theorem 1.4.3]{AlarconForstnericLopez2021Book}) and after possibly shrinking $\Omega'$, we have that for almost every $t \in \Omega'$ the map $\delta H (\cdot, \cdot, \cdot, t)$ is transverse to $\Lambda$ on $I \times (\om \times \om \setminus U)$ and hence, by dimension reasons
\begin{equation}\label{eq:notlambda}
 \delta  H ( I \times (\om \times \om \setminus U), t) \cap \Lambda = \varnothing.
\end{equation}
Choosing $t_0 \in \Omega'$ sufficiently close to zero we get a holomorphic Legendrian curve
\[
\wt X  = (\wt X_1, \wt X_2, \wt X_3)  := H (\cdot, t_0) : \om \to \C^3
\]
which is as close as desired to $X$ on the $\Cscr^1 (\om)$-norm, see \ref{core} and \ref{legendrian}. In particular, we may choose $t_0$ so that \eqref{eq:u} holds for $\wt X$. Indeed, if $(p,q)\in U \setminus D_{\om}$, then $(\wt X_1 (p), \wt X_2(p)) \not = (\wt X_1 (q), \wt X_2 (q))$ provided that $\wt X$ is close enough to $X$. 
Taking also \eqref{eq:notlambda}  into account, it follows that 
\[
\delta \wt X ( I \times (\om \times \om \setminus D_{\om}) ) \cap \Lambda = \varnothing.
\]
In view of \eqref{eq:dD}, $\wt X$ satisfies the conclusion of Lemma \ref{lemma:theta}.
\end{proof}

%
%
\section{Proof of Theorem \ref{th:mtleg}}\label{sec:proofleg}
We prove Theorem \ref{th:mtleg} as a consequence of the following result, which generalizes \cite[Theorem 4.2]{Alarcon2017JGA}; recall Remark \ref{rk:periodic}.
%
%
\begin{theorem}\label{th:mtleggen}
Let $M$ be an open Riemann surface, $K \subset M$ a smoothly bounded Runge compact domain, and $X = (X_1, X_2, X_3): K \to \C^3$ a Legendrian curve of class $\Ascr^1 (K)$. Then $X$ may be approximated in the $\Cscr^1 (K)$-norm by holomorphic Legendrian embeddings $\wt X = (\wt X_1, \wt X_2, \wt X_3): M \to \C^3$ such that
\begin{enumerate}[label = \rm (\roman*)]
\item \label{conc1} The function 
$
\max \{ | \wt X_1 |, - \Re \wt X_3\} : M \to \R_+ =  [0, \infty)
$
is proper; where $\Re$ denotes the real part. 

\item \label{conc2}  $	\delta^{\Dcal} \wt X  :  [0, 2 \pi]  \times M \times  M   \to \C^3$
does not vanish on $[0, 2 \pi] \times (M \times M \setminus D_M)$.
\end{enumerate}
\end{theorem}

Property \ref{conc1} is a technical condition that guarantees properness of the map $\wt F = \Dcal \circ \wt X : M \to \SL$. 
It already appears in \cite{Alarcon2017JGA} and we just apply the properness lemma given in this reference to guarantee that it holds now. 
The novelty in Theorem \ref{th:mtleggen} is property \ref{conc2}, which is obtained from a recursive application of Lemma \ref{lemma:theta}. 
%
%
\begin{proof}[Proof of Theorem \ref{th:mtleg} assuming Theorem \ref{th:mtleggen}]
As before, set $I =  [0, 2 \pi]$.
Let $M$ be an open Riemann surface. Theorem \ref{th:mtleggen} furnishes a holomorphic Legendrian curve $X = (X_1, X_2, X_3)  : M \to \C^3$ such that 
\begin{enumerate}[label = \rm (\roman*)]
	\item \label{conc1app}$\max \{ | X_1|, - \Re  X_3\} : M \to \R_+$
	is proper and
		
	\item \label{conc2app}$\delta^{\Dcal}X :  I \times M \times  M  \to \C^3$ does not vanish on $I \times (M \times M \setminus D_M)$.
\end{enumerate}
Consider the holomorphic Legendrian curve 
\[
	F := 	\Dcal \circ X =\left(\begin{matrix}
			e^{-X_3} & X_1 e^{X_3}\\
			X_2 e^{-X_3} &  (1 + X_1 X_2) e^{X_3}
		\end{matrix}\right) : M \to \SL,
\]
recall \eqref{eq:d}.
By the same argument as in \cite[Proof of Theorem 1.1]{Alarcon2017JGA}, this map is proper in $\SL$ by \ref{conc1app}. 
Moreover,  $\Delta F : I \times M \times M \to \C^4$ does not vanish on $I \times (M \times M \setminus D_M)$ by \ref{conc2app}. Indeed, suppose that there exists $(\theta, p,q) \in I \times (M \times M \setminus D_M)$ such that 
\[
\Delta F (\theta, p, q) = F(q) - F(p) A_{\theta} = 0,
\]
recall the notation in \eqref{eq:u1}. This is equivalent to 
\[
X_1 (q) = X_1 (p), \quad X_2 (q)= X_2 (p), \quad e^{-X_3 (q)} =  e^{\igot \theta} e^{- X_3 (p)},
\]
which is impossible in view of \ref{conc2app} and the definition of $\delta^{\Dcal}$, see \eqref{eq:thetadifleg}.
\end{proof}

%
%
\begin{proof}[Proof of Theorem \ref{th:mtleggen}]
Again set $I = [0, 2 \pi]$.
We may assume that $X$ extends to a holomorphic Legendrian curve on an open neighborhood of $K$ and has no constant component function by \cite[Theorem 5.1]{AlarconForstnericLopez2017CM} (see also \cite[Theorem 2.3]{Alarcon2017JGA}). Thus we may assume, up to slightly enlarging $K$, that there is a number $\rho_0 >0$ such that
\begin{equation}\label{eq:max}
\max \{  |X_1|, - \Re X_3\} > \rho_0 \text{ on } b K.
\end{equation}
Let 
\begin{equation}\label{eq:mn}
	K \Subset M_1 \Subset M_2 \Subset \cdots \subset \bigcup_{n\in\n} M_n = M
\end{equation}
be an exhaustion of $M$ by connected smoothly bounded Runge compact domains. Consider also 
\[
  U_1 \Supset U_2 \Supset \cdots \supset D_M
\]
a decreasing family of smoothly bounded domains in $M \times M$ such that $M_n \times M_n \setminus U_n$ is a smoothly bounded compact domain for every $n \in \N$ and
\begin{equation}\label{eq:mmun}
M_1 \times M_1 \setminus U_1 \Subset M_2 \times M_2 \setminus U_2 \Subset \cdots \subset M \times M \setminus D_M
\end{equation}
is an exhaustion of $M \times M \setminus D_M$. 

Fix $\varepsilon_0 >0$. In view of \eqref{eq:max} and setting $K = M_0$, $X = X^0|_{M_0} : M_0 \to \C^3$, we may apply \cite[Lemma 4.3]{Alarcon2017JGA} in a standard inductive way and obtain sequences of real numbers $\varepsilon_n >0$ and Legendrian curves $X^n =(X_1^n, X_2^n, X_3^n) : M_n \to \C^3$ of class $\Ascr^1 (M_n)$ such that the following conditions hold for every $n \in \N$. 
	\begin{enumerate}[label = \rm (\alph*$_n$)]
	 \item \label{app} $\| X^n - X^{n-1} \|_{1, M_{n-1}} < \varepsilon_{n-1}$.
	 \item \label{prop} $\max \{ |  X^n_1, - \Re  X^n_3 \} > n - 1$ on $M_n \setminus \mathring M_{n-1}$.
	 \item \label{propb} $\max \{ |  X_1^n|, - \Re X_3^n \} > n$ on $b M_n$. 
	 \newcounter{an}\setcounter{an}{\value{enumi}}
	\end{enumerate}
Moreover, by Lemma \ref{lemma:theta}, we may choose the aforementioned sequences such that
\begin{enumerate}[label = \rm (\alph*$_n$)]
	\setcounter{enumi}{\value{an}}
	\item \label{deltat} $\delta^{\Dcal} X^n: I \times M_n \times M_n \to \C^3$ does not vanish on $I \times( M_n \times M_n \setminus D_{M_n})$.
	 \item \label{eps}$0 < \varepsilon_n < \varepsilon_{n-1} /2$ and if $Y : M_n \to \C$ is a continuous function with $\| X^n - Y\|_{0, M_n} < 2 \varepsilon_n$, then $\delta^{\Dcal} Y$ does not vanish on $I \times (M_n \times M_n \setminus U_n)$.
\end{enumerate}
Indeed, suppose that we have $X^n$ satisfying \ref{app}--\ref{propb}. By a small deformation provided by Lemma \ref{lemma:theta}, we may assume that $X^n$ also satisfies \ref{deltat}. 
By compactness there exists $a_n\in \R_+$ such that
\[
	| \delta^{\Dcal} X^n| > a_n \text{ on }   I \times (M_n \times M_n \setminus U_n),
\] 
Let $Y =(Y_1, Y_2, Y_3) : M_n \to \C$ be a continuous function with $\| X^n - Y\|_{0, M_n} < 2 \varepsilon_n$. There is another constant $b_n>0$ that only depends on $X^n : M_n \to \C^3$ such that 
\[
|\delta^{\Dcal} Y  - \delta^{\Dcal}X^n  | < b_n \varepsilon_n \text{ on }   I \times M_n \times M_n,
\]
see \eqref{eq:thetadifleg} and note that $|e^{X_3^n} - e^{Y_3}| \le c_n |X_3^n - Y_3|$ on $M_n$ for certain constant $c_n >0$.
Consequently,
\[
|\delta^{\Dcal} Y | > | \delta^{\Dcal}X^n | - b_n \varepsilon_n > a_n - b_n \varepsilon_n \text{ on } I \times (M_n \times M_n \setminus U_n),
\]
so it suffices to choose 
\[
\varepsilon_n < \min  \{ \varepsilon_{n-1} /2, \frac{a_n}{b_n}   \}
\]
to guarantee condition \ref{eps}.

By \ref{app}, \ref{eps}, and \eqref{eq:mn}, there exists a limit holomorphic Legendrian curve
\[
\wt X = (\wt X_1, \wt X_2, \wt X_3) : M= \bigcup_{n\in\n} M_n \to \C^3
\]
such that 
\[
	\| \wt X - X^n \|_{1, M_n} < 2 \varepsilon_n \text{ for every } n \in \{0\} \cup \N.
\]
In particular, $\wt X$ may be chosen as close as desired to $X^0 = X$ on the $\Cscr^1 (K)$-norm, recall that $K=M_0$ and that $\varepsilon_0$ was arbitrarily chosen.
By the previous equation and \ref{eps} we obtain that $\delta^{\Dcal}\wt X: I \times M \times M \to \C^3$ does not vanish anywhere on $I \times (M_n \times M_n \setminus U_n)$, for every $n\in \N$.
Taking \eqref{eq:mmun} into account, \ref{conc2} holds.
Moreover, choosing each $\varepsilon_n$ sufficiently small, property \ref{prop} ensures that  $\max \{ | \wt X_1 |, - \Re  \wt X_3 \} : M  \to \R_+$ is a proper map (see \cite[Proof of Theorem 4.2]{Alarcon2017JGA} for the precise argument), so \ref{conc1} is satisfied. 
This concludes the proof.
\end{proof}

%
%
\subsection*{Acknowledgements}
Research supported by the State Research Agency (AEI) via the grant no. 
 PID2023-150727NB-I00 funded by MICIU/AEI/10.13039/501100011033, Spain.


	
	\noindent Jorge Hidalgo
	
	\noindent Departamento de Geometr\'{\i}a y Topolog\'{\i}a,  Universidad de Granada, Campus de Fuentenueva s/n, E--18071 Granada, Spain.
	
	\noindent  e-mail: {\tt jorgehcal@ugr.es}

\end{document}